\documentclass[a4paper,11pt]{amsart}

\usepackage[left=2.5cm,right=2.5cm]{geometry}
\usepackage[english]{babel}
\usepackage[utf8]{inputenc}
\usepackage{amsmath, amsfonts, amssymb,amsthm,graphicx}

\newtheorem{Theorem}{Theorem}[section]

\newtheorem{Prop}[Theorem]{Proposition}

\theoremstyle{definition}

\newtheorem{Rmks}[Theorem]{Remarks}

\newcommand{\Fq}[1][]{\mathbb{F}_{q^{#1}}}

\newcommand{\Z}{\mathbb{Z}}
\def\ie{\textit{i.e.}}

\title[Local cyclicity of isogeny classes of abelian varieties]{Local cyclicity of isogeny classes of abelian varieties defined over finite fields}
\author{Alejandro J. Giangreco-Maidana}
\address{Aix Marseille Universit\'e, CNRS, Centrale Marseille, I2M UMR 7373, 13453 Marseille, France}
\email{ajgiangreco@gmail.com}

\begin{document}


\begin{abstract}
For a prime number $\ell$, an isogeny class $\mathcal{A}$ of abelian varieties is called $\ell$\emph{-cyclic} if every variety in $\mathcal{A}$ have a cyclic $\ell$-part of its group of rational points. More generally, for a finite set of prime numbers $\mathcal{S}$, $\mathcal{A}$ is said to be $\mathcal{S}$\emph{-cyclic} if it is $\ell$-cyclic for every $\ell\in\mathcal{S}$. We give lower and upper bounds on the fraction of $\mathcal{S}$-cyclic $g$-dimensional isogeny classes of abelian varieties defined over the finite field $\Fq$, when $q$ tends to infinity. 
\end{abstract}

\maketitle

\section{Introduction}
The group of rational points $A(k)$ of an abelian variety $A$ defined over a finite field $k=\Fq$ is a finite abelian group, and it has theoretical and practical interests. More precisely, the group structure of $A(k)$ and some statistics are research topics in the literature. 

The structure of all possible groups for elliptic curves defined over finite fields was independently discovered in \cite{SCHOOF1987183}, \cite{ruck1987note},\cite{tsfasman1985group} and \cite{voloch1988note}. For higher dimensions, Rybakov gives in \cite{Rybakov2010} a very explicit description of all possible groups of rational points of an abelian variety in a given isogeny class, his result is formulated in terms of the characteristic polynomial of the isogeny class.

Cyclic subgroups of the group of rational points $A(k)$ of an abelian variety $A$ defined over a finite field $k$ are suitable for multiple applications, in particular for cryptography, where the Discrete Logarithm Problem is exploited. Elliptic curves with cyclic groups of rational points are of special interest, and many algorithms chooses it randomly. Thus, knowing such statistics can be useful. Statistics on the fraction of isomorphism classes of elliptic curves that are cyclic were explored in \cite{VLADUT199913} and \cite{VLADUT1999354}, and for higher dimensions, this was explored in \cite{ajgiangreco2018cavff}.

More precisely, in \cite{ajgiangreco2018cavff} we studied the cyclicity of isogeny classes of abelian varieties defined over finite fields. In particular, if $\mathcal{S}$ is a finite set of prime numbers, from these results we can easily deduce a nice criterion to know if an isogeny class $\mathcal{A}$ is $\mathcal{S}$\emph{-cyclic}, \ie, for every variety $A\in\mathcal{A}$ and every $\ell\in \mathcal{S}$, the $\ell$-part $A(k)_\ell$ of $A(k)$ is cyclic. This criterion is based on the characteristic polynomial of the isogeny class $\mathcal{A}$. Clearly, this result is meaningful only if $A(k)_\ell$ is nontrivial. As a consequence of Tate theorem (\cite{tate1966}), the cardinality of $A(k)$ is an invariant of the isogeny class, and thus, the property of $A(k)_\ell$ to be or not trivial is an invariant of the isogeny class as well.

Using the strategy proposed in \cite{DIPIPPO1998426}, Holden proved in \cite{Holden2004} that the fraction of $k$-isogeny classes with nontrivial $\ell$-part $A(k)_\ell$ is asymptotically $1/\ell$, when $q$ tends to infinity.

In this paper, we combine both results (from \cite{ajgiangreco2018cavff} and \cite{Holden2004}) to provide lower and upper (asymptotic) bounds, in Theorem \ref{th:S-cyclicity}, on the fraction of $k$-isogeny classes of abelian varieties that are $\mathcal{S}$-cyclic, among those $k$-isogeny classes with a nontrivial $\ell$-part, for at least one $\ell\in\mathcal{S}$.

The rest of this paper is organized as follows: in Section \ref{sec:Preliminaries} we briefly recall some general facts about abelian varieties over finite fields and we state our result more precisely; in Section \ref{sec:proof} we prove Theorem \ref{th:S-cyclicity}; finally, in Section \ref{sec:appendix}(Appendix) we prove some propositions used in the proof of Theorem \ref{th:S-cyclicity}.

\section{Preliminaries and Statement of the results}\label{sec:Preliminaries}
Let establish the results more precisely. From the Honda-Tate theorem (\cite{tate1966,honda1968}), every isogeny class $\mathcal{A}$ of abelian varieties defined over a finite field is uniquely determined by its Weil polynomial $f_\mathcal{A}(t)\in \Z [t]$, the characteristic polynomial of its Frobenius endomorphism acting on its Tate module. The cardinality of the group of rational points of a variety $A$ in $\mathcal{A}$ equals $f_\mathcal{A}(1)$, thus is an invariant of the isogeny class.
Let $q=p^r$ be a power of a prime and denote by $\mathcal{I}(q,g)$ the set of $g$-dimensional isogeny classes $\mathcal{A}$ of abelian varieties defined over the finite field $\Fq$.
 
For a fixed and finite set of prime numbers $\mathcal{S}$ (even containing $p$), we say that an isogeny class $\mathcal{A}$ is $\mathcal{S}$-cyclic if all its varieties have a cyclic $\ell$-part of its group of rational points for all $\ell\in\mathcal{S}$. 
For more details on the general theory of abelian varieties see for example \cite{mumford1970abelian}, and for precise results over finite fields, see \cite{Waterhouse1969}.

Following the notation in \cite{Holden2004}, we define the following cardinalities
\begin{itemize}
\item $I_\mathcal{S}(q,g)=\#\{ \mathcal{A}\in  \mathcal{I}(q,g): \ell|f_\mathcal{A}(1)$ for some prime $\ell\in\mathcal{S}\}$ (those with $(A(k))_\ell$ non trivial) and,
\item $I_\mathcal{S}^n(q,g)=\#\{ \mathcal{A}\in  \mathcal{I}(q,g): \ell|f_\mathcal{A}(1)$ with $\mathcal{A}$ non $\mathcal{S}$-cyclic $\}$. 
\end{itemize}

Then, our result provides the proportion of $\mathcal{S}$-cyclic isogeny classes:
\begin{Theorem}\label{th:S-cyclicity}
There exist a $q_0=q_0(\mathcal{S})$, such that if $q\geq q_0$, we have
\begin{align}
 1-\frac{1-\sigma_2(\mathcal{S})}{1-\sigma_1(\mathcal{S})} \leq \frac{ I_\mathcal{S}(q,g) - I_\mathcal{S}^n(q,g)}{I_\mathcal{S}(q,g)} \leq 1-\frac{1-\sigma_3(\mathcal{S})}{1-\sigma_1(\mathcal{S})},
\end{align}
where
\begin{align*}
\sigma_i(\mathcal{S})=\prod_{\ell \in\mathcal{S}} \left(  1-\frac{1}{\ell^i} \right).
\end{align*}
\end{Theorem}

\begin{Rmks}
$ $
\begin{enumerate}
\item If we consider the sets $\mathcal{S}(N)=\{\ell$ prime $: \ell\leq N\}$, then when $N\rightarrow\infty$, $\sigma_i(\mathcal{S}(N))\rightarrow 1/\zeta(i)$, where $\zeta$ is the Riemann zeta function. Thus, since $\zeta(1)=\infty$, the bounds can be simplified
\begin{align*}
0.6 \approx\frac{1}{\zeta(2)}  \leq \frac{ I_{\mathcal{S}(N)}(q,g) - I_{\mathcal{S}(N)}^n(q,g)}{I_{\mathcal{S}(N)}(q,g)} \leq  \frac{1}{\zeta(3)} \approx 0.833.
\end{align*}
Computational results show that these limits start in $0.5$ and $0.75$ (when $N=2$) and establish quickly at $0.57$ and $0.815$ (around $N=557$).
\item When we take $\mathcal{S}=\{\ell\}$, these bounds are $1-1/\ell$ and $1-1/\ell^2$. Moreover, it is not hard to deduce from the proof of Theorem \ref{th:S-cyclicity}, that we have exact limits
\begin{align*}
\lim_{q\rightarrow\infty, \ell|q-1} \frac{I_\mathcal{S}(q,g) - I_\mathcal{S}^n(q,g)}{I_\mathcal{S}(q,g)}=\frac{\ell -1}{\ell},\\
\lim_{q\rightarrow\infty, \ell\nmid q-1} \frac{I_\mathcal{S}(q,g)- I_\mathcal{S}^n(q,g)}{I_\mathcal{S}(q,g)}=\frac{\ell^2 -1}{\ell^2}.
\end{align*}
Considering the result in \cite{Holden2004}, this means that the fraction of $g$-dimensional $\ell$-cyclic isogeny classes (with nontrivial $\ell$-part) among all $g$-dimensional isogeny classes is $\frac{1}{\ell}\left(\frac{\ell -1}{\ell}\right)$ or $\frac{1}{\ell}\left(\frac{\ell^2 -1}{\ell^2}\right)$, when $q$ tends to infinity among such $q$ with $\ell|(q-1)$ or $\ell\nmid (q-1)$, respectively. 
\item From \cite{Holden2004} we can see that the fraction of isogeny classes with non-trivial $\ell$-part for at least one $\ell\in\mathcal{S}$ is $1-\sigma_1(\mathcal{S})$. \end{enumerate}
\end{Rmks}

\section{The proof}\label{sec:proof}
Here we prove Theorem \ref{th:S-cyclicity}. Also, for the clarity of the text, the propositions used for the proofs of theorems are proved in the Appendix.

Our procedure is based on the method of counting isogeny classes of abelian varieties introduced by DiPippo and Howe in \cite{DIPIPPO1998426}, after having done some partitions in particular lattices, similar as in \cite{Holden2004}.

We will consider \emph{rectilinear lattices}, by which we mean lattices that have a rectilinear fundamental domain with edges parallel to the coordinate axes. The \emph{covolume} of a lattice $\Lambda$ is the volume of a fundamental domain R of $\Lambda$, and the \emph{mesh} of $\Lambda$ is the length of the longest edge of R.

Every isogeny class $\mathcal{A}$ is defined by its Weil polynomial $f_\mathcal{A}$, which has the general form
\begin{equation}
f_\mathcal{A}(t)=t^{2g}+a_1 t^{2g-1}+\dots + a_g t^g + a_{g-1} q t^{g-1}+\dots a_1 q^{g-1} t + q^{g} \in \mathbb{Z}[t],
\end{equation}
and for simplicity, we write it as $f_{q,(a_1,\dots,a_{g})}(t)$. Not every polynomial of the previous form defines an isogeny class, but we know ``almost all'' polynomials that are Weil polynomials of some isogeny class, and we will clarify it in the next paragraph. For a polynomial of the previous form $f_{q,\mathbf{a}}(t)$, we associate another polynomial
\begin{align*}
h_{\mathbf{b}}(t)=f_{q,\mathbf{a}}(\sqrt{q} t)/ q^g=(t^{2g}+1)+b_1(t^{2g-1}+t) + \dots + b_{g-1}(t^{2g+1}+t^{2g-1})+b_g t^g \in \mathbb{R}[t].
\end{align*}

We consider three types of lattices. The first one, $\Lambda_q$, is generated by the vectors $q^{-i/2}e_i$, where $e_1,\dots,e_g$ denote the standard basis vectors of $\mathbb{R}^g$. From DiPippo and Howe \cite{DIPIPPO1998426}, it follows that if an isogeny class has $f_{q,\mathbf{a}}(t)$ as its Weil polynomial, then $(a_1 q^{-1/2},\dots, a_g q^{-g/2})\in \Lambda_q \cap V_g$, where $V_g$ is the set of vectors $\mathbf{b}=(b_1,\dots,b_g) \in \mathbb{R}^g$ such that the associated $h_{\mathbf{b}}(t)$ has all its roots on the unit circle and all its real roots occur with even multiplicity. The second lattice, $\Lambda'_q$, is generated by the vectors $q^{-1/2}e_1,\dots, q^{-(g-1)/2}e_{g-1}$ and $pq^{-g/2}e_g$. Then, there is a bijection between isogeny classes of ordinary varieties with Weil polynomial $f_{q,\mathbf{a}}(t)$, and $(\Lambda_q\setminus \Lambda'_q) \cap V_g$, given by $\mathbf{a}\mapsto (a_1 q^{-1/2},\dots, a_g q^{-g/2})$. Finally, the third lattice, $\Lambda''_q$, is generated by the vectors $q^{-1/2}e_1,\dots, q^{-(g-1)/2}e_{g-1}$ and $sq^{-g/2}e_g$, where $s$ is the smallest power of $p$ such that $q$ divides $s^2$. Then every non-ordinary isogeny class has its Weil polynomial $f_{q,\mathbf{a}}(t)$ such that $(a_1 q^{-1/2},\dots, a_g q^{-g/2})\in \Lambda''_q \cap V_g$. In order to measure the sizes of these kind of sets (we will refine these lattices later), we will use Proposition $2.3.1$ of \cite{DIPIPPO1998426} in a generalized form:

\begin{Prop}[see \cite{DIPIPPO1998426}, and \cite{Holden2004} for its generalized form] \label{prop:lattice}
Let $g>0$ be an integer and let $\Lambda \subset \mathbb{R}^g$ be a rectilinear lattice (possibly shifted) with mesh $d$ at most $D$. Then we have
\begin{equation*}
\left| \#(\Lambda \cap V_g) - \frac{\text{volume } V_g}{\text{covolume } \Lambda} \right| \leq c(g,D)\frac{d}{\text{covolume } \Lambda}
\end{equation*}
for some constant $c(g,D)$ depending only on $g$ and $D$ which can be explicitly computed. (We will not need the explicit computation in this paper.)
\end{Prop}

Denote by $F=F(\mathcal{S})$ the product of primes contained in $\mathcal{S}$. 
%If $\mathcal{P}$ denotes the set of prime numbers, 
From Theorem 2.2 in \cite{ajgiangreco2018cavff}, we have that $\mathcal{A}$ is not cyclic if and only if $(\widehat{f(1)},f'(1))>1$, where $\widehat{z}$ is the integer $z$ divided by the product of its different prime divisors. Locally, $\mathcal{A}$ is not $\ell$-cyclic if and only if $\ell |(\widehat{f(1)},f'(1))$.
Then, we have a partition of the form
\begin{gather}
I_{\mathcal{S}}(q,g)=\sum I_\mathbf{m}(q,g), \label{eq:partS}\\
I_{\mathcal{S}}^n(q,g)=\sum I_\mathbf{m}(q,g), \label{eq:partSn}
\end{gather}
where the summations on equations \ref{eq:partS} and \ref{eq:partSn} are taken, respectively, over
\begin{itemize}
\item all $\mathbf{m}\in (\mathbb{Z}/F^2\mathbb{Z})^g$, such that $f_{q,\mathbf{m}}(1)$ is not invertible in $\Z/F^2\Z$, \ie, such that $A(\Fq)_\ell$ is not trivial for some $\ell\in \mathcal{S}$, and
\item all $\mathbf{m}\in (\mathbb{Z}/F^2\mathbb{Z})^g$, such that $\ell |(\widehat{f(1)},f'(1))$ for some $\ell \in \mathcal{S}$, \ie, $A(\Fq)$ have some non cyclic component;
\end{itemize}
and where $I_\mathbf{m}(q,g)$ is the number of isogeny classes of $g$-dimensional abelian varieties over $\Fq$ with Weil polynomial $f_{q,\mathbf{a}}$ for some $\mathbf{a} \in \mathbb{Z}^g$ such that $\mathbf{a} \equiv \mathbf{m} \pmod{F^2}$. The number of terms on the right side of equations above will be important, and we will see it in a proposition below.

Now we refine the latices mentioned above, and we will calculate their sizes using Proposition \ref{prop:lattice}. 
If $\Lambda_\mathbf{m}$ is the lattice generated by the vectors $F^2 q^{-i/2}e_i$ and then shifted by $\sum_i m_i q^{-i/2}e_i$, we consider the lattices $\Lambda'_\mathbf{m}=\Lambda_\mathbf{m}\cap \Lambda'_q$ and $\Lambda''_\mathbf{m}=\Lambda_\mathbf{m}\cap \Lambda''_q$. Set $G=g(g+1)/4$, then $\Lambda_\mathbf{m}$ has covolume $F^{2g} q^{-G}$ and mesh $F^2 q^{-1/2}$; $\Lambda'_\mathbf{m}$ has covolume $F^{2g} p q^{-G}$ and mesh $F^2 q^{-1/2}$ unless $g=2$ and $q=p$ in which case it has mesh $F^2$; and $\Lambda''_\mathbf{m}$ has covolume $F^{2g} s q^{-G}$ and mesh at most $F^2$.

From Proposition \ref{prop:lattice} we get (see Appendix for the proof)
\begin{Prop}\label{prop:Im_F}
\begin{gather*}
L(q,g,F) F^{-2g} \leq I_\mathbf{m}(q,g) \leq R(q,g,F) F^{-2g},
\end{gather*}
with
\begin{align*}
L(q,g,F)=\left(v_g r(q) q^G - 2F^2 c(g,F^2)q^{G-1/2}\right) \quad\text{and} \\ %\quad
R(q,g,F)=\left(v_g r(q) q^{G} + (v_g + 3F^2 c(g,F^2))q^{G-1/2} \right),
\end{align*}
and where $v_g$ is the volume of $V_g$ and $r(q)=1-1/p$.
\end{Prop}

Note that $L/R$ tends to  $1$ when $q$ tends to infinity. Also, we can show that (see Appendix for the proof)

\begin{Prop}\label{prop:numbSol}
$ $
\begin{itemize}
\item the number of terms on the RHS in equation \ref{eq:partS} is 
\[
F^{2g} (1-\sigma_1(\mathcal{S})), \quad\text{and},
\] %F^{2g}-\varphi(F^{2g})=
\item the number of terms on the RHS in equation \ref{eq:partSn} is between \[
F^{2g}(1-\sigma_3(\mathcal{S})) \qquad \text{and} \qquad F^{2g}(1-\sigma_2(\mathcal{S})).
\]
\end{itemize}
\end{Prop}

Thus we get
\begin{align}
(1-\sigma_1(\mathcal{S})) L(q,g,F) \leq I_\mathcal{S}(q,g) &\leq (1-\sigma_1(\mathcal{S})) R(q,g,F), \qquad\text{and} \label{eq:S_bounds}\\
(1-\sigma_3(\mathcal{S})) L(q,g,F) \leq I_\mathcal{S}^n(q,g) &\leq (1-\sigma_2(\mathcal{S})) R(q,g,F), \label{eq:nS_bounds}
\end{align}

then, when $q$ tends to infinity
\begin{gather}
\frac{1-\sigma_3(\mathcal{S})}{1-\sigma_1(\mathcal{S})} \leq \frac{I_{\mathcal{S}}^n(q,g)}{I_{\mathcal{S}}(q,g)} \leq \frac{1-\sigma_2(\mathcal{S})}{1-\sigma_1(\mathcal{S})}
\end{gather}

which gives the result. Theorem \ref{th:S-cyclicity} is thus proved.

\section{Appendix}\label{sec:appendix}
\begin{proof}[Proof of Proposition \ref{prop:Im_F}]
Here we prove the $\mathcal{S}$-case. For simplicity, we omit the subscript $\mathbf{m}$ from $I_\mathbf{m}(q,g)$, $\Lambda_\mathbf{m}$, $\Lambda'_\mathbf{m}$ and $\Lambda''_\mathbf{m}$, and we write $v$ instead of $v_g$. Concerning these three lattices, we summarize its values in the following table

\begin{center}
  \begin{tabular}{ | c | c | c | c | c |}
    \hline
    & covolume & $d$ & $D$ & $d/covolume$ \\ \hline
    $\Lambda$ & $F^{2g} q^{-G}$ & $F^2 q^{-1/2}$ & $F^2$ & $F^{2-2g}q^{G-1/2}$ \\ \hline
    $\Lambda'$ & $F^{2g} p q^{-G}$ & $F^2 q^{-1/2}$ & $F^2$ & $F^{2-2g}q^{G-1/2}/p$ \\ \hline
    $\Lambda''$ & $F^{2g} s q^{-G}$ & $F^2$ & $F^2$ & $F^{2-2g}q^{G}/s$\\
    \hline
  \end{tabular}
\end{center}

We apply Proposition \ref{prop:lattice}, which in the general form says
\begin{align*}
\frac{\text{volume } V_g}{\text{covolume } \Lambda} -c(g,D)\frac{d}{\text{covolume } \Lambda} &\leq \#(\Lambda \cap V_g) \leq \frac{\text{volume } V_g}{\text{covolume } \Lambda} +c(g,D)\frac{d}{\text{covolume } \Lambda},
\end{align*}

thus, for every lattice we get, respectively
\begin{align*}
vF^{-2g}q^G - c(g,F^2)F^{2-2g}q^{G-1/2} &\leq \#(\Lambda \cap V_g) \leq vF^{-2g}q^G + c(g,F^2)F^{2-2g}q^{G-1/2} \\
vF^{-2g}q^G/p - c(g,F^2)F^{2-2g}q^{G-1/2}/p &\leq \#(\Lambda' \cap V_g) \leq vF^{-2g}q^G/p + c(g,F^2)F^{2-2g}q^{G-1/2}/p \\
vF^{-2g}q^G/s - c(g,F^2)F^{2-2g}q^G/s &\leq \#(\Lambda'' \cap V_g) \leq vF^{-2g}q^G/s + c(g,F^2)F^{2-2g}q^G/s 
\end{align*}
From the description of DiPippo and Howe,
\[
\#( (\Lambda \setminus \Lambda')\cap V_g) \leq I(q,g)\leq \#((\Lambda\cup\Lambda''\setminus \Lambda')\cap V_g)
\]
Thus, 
\begin{align*}
I(q,g) &\geq vF^{-2g}q^G - c(g,F^2)F^{2-2g}q^{G-1/2}  - \left( vF^{-2g}q^G/p + c(g,F^2)F^{2-2g}q^{G-1/2}/p \right)\\
&=vF^{-2g}q^G r(q) - c(g,F^2)F^{2-2g}q^{G-1/2}(1+1/p)\\
&\geq vF^{-2g}q^G r(q) - 2c(g,F^2)F^{2-2g}q^{G-1/2}.
\end{align*}
Also (knowing that $s$ is roughly $\sqrt{q}$ and $\sqrt{q}\leq s$),
\begin{align*}
I(q,g) &\leq  vF^{-2g}q^G + c(g,F^2)F^{2-2g}q^{G-1/2} - \left( vF^{-2g}q^G/p - c(g,F^2)F^{2-2g}q^{G-1/2}/p  \right) + \\ &+ vF^{-2g}q^G/s + c(g,F^2)F^{2-2g}q^G/s \\
&=v F^{-2g} r(q) q^G + vF^{-2g}q^G/s +c(g,F^2) F^{2-2g} \left( q^{G-1/2} +q^{G-1/2}/p + q^G/s\right) \\
&\leq v F^{-2g} r(q) q^G + vF^{-2g}q^{G-1/2} +c(g,F^2) F^{2-2g} \left( q^{G-1/2} +q^{G-1/2}/p + q^{G-1/2}\right) \\
&= v F^{-2g} r(q) q^G + v F^{-2g}q^{G-1/2} +c(g,F^2) F^{2-2g} q^{G-1/2} \left( 1 + 1/p + 1\right) \\
&\leq v F^{-2g} r(q) q^G + (v+3 F^2 c(g,F^2)) F^{-2g} q^{G-1/2}
\end{align*}
Thus, the proposition is proved.
\end{proof}

\begin{proof}[Proof of the Proposition \ref{prop:numbSol}]
Knowing the number of terms of these equations is equivalent to knowing the number of solutions of certain systems of equations. 

The first item follows from the fact that the number of solutions is $F^{2g-2}(F^2-\varphi(F^2))$, where $\varphi$ is the Euler's totient function, since for every $\mathbf{n}\in (\mathbb{Z}/F^2\mathbb{Z})^{g-1}$, there are exactly $F^2-\varphi(F^2)$ of $x$ satisfying that $f_{q,(\mathbf{n},x)}(1)$ is not invertible in $\Z/F^2\Z$.  

For the second item , it means, there exist a prime $\ell\in \mathcal{S}$ such that
\begin{align*}
f_{q, \mathbf{m}}(1)\equiv i\ell^2 \pmod{F^2}\\
f'_{q, \mathbf{m}}(1)\equiv j\ell \pmod{F^2}
\end{align*}
We can study locally by the well known isomorphism
\begin{align*}
(\Z/F^2\Z)^{g} \cong \prod_{\ell|F} (\Z/\ell^2\Z)^{g}
\end{align*}
In this case, the number of solution is between $\ell^{2g-3}$ and $\ell^{2g-2}$ (see below for the proof). Then the result follows from the inclusion-exclusion formula, since we have a solution of the problem if and only if we have a solution on at least one component. Now, we study the problem locally, and we split it in two cases:\\
When $\ell\nmid q-1$, for every $\mathbf{n}\in (\mathbb{Z}/\ell^2\mathbb{Z})^{g-2}$, the system
\begin{align*}
f_{q, (\mathbf{n},x,y)}(1) &\equiv 0 \pmod{\ell^2}\\
f'_{q,(\mathbf{n},x,y)}(1) &\equiv i\ell \pmod{\ell^2}, i=0,1,\dots,\ell-1
\end{align*}
has exactly $\ell$ solutions (exactly one for every $i$ on the second equation), since we can write 
\begin{align*}
f_{q, (\mathbf{n},x,y)}(1)&=h_1(\mathbf{n})+(q+1)x+y \qquad\text{ and}\\
f'_{q, (\mathbf{n},x,y)}(1)&=h_2(\mathbf{n})+ [g(q+1)+1-q]x + gy
\end{align*}
which determinant is $q-1 \pmod{\ell^2}$.\\
When $\ell | q-1$, we have
\begin{align*}
f_{q,\mathbf{a}}(1)&\equiv 2+2a_1 + \dots + 2a_{g-1} + a_g &\pmod{\ell}\\
f'_{q,\mathbf{a}}(1)&\equiv g(2+2a_1 + \dots + 2a_{g-1} + a_g) \equiv g f_{q, \mathbf{a}}(1) &\pmod{\ell}
\end{align*}
\ie, $f' \equiv gf \pmod{\ell}$, which means $f'-gf\equiv i\ell \pmod{\ell^2}$ for some $i$, thus, if $\mathbf{a}$ is a solution for the first equation, $f'\equiv i\ell \pmod{\ell^2}$, which means $f'\equiv 0\pmod{\ell}$. Now, for every $\mathbf{n}\in (\mathbb{Z}/\ell^2\mathbb{Z})^{g-1}$, we have exactly one solution for the first equation, and which is automatically a solution for the second equation by the previous explanation. Thus, the result follows.
\end{proof}

\section*{Acknowledgement}
I would like to thank my advisor Serge Vl\u{a}du\c{t} for very fruitful discussions and for his very useful comments and suggestions.

\bibliography{ajgiangreco-LCICAVFF}

\begin{thebibliography}{10}

\bibitem{DIPIPPO1998426}
{\sc S.~A. DiPippo and E.~W. Howe}, {\em Real polynomials with all roots on the
  unit circle and abelian varieties over finite fields}, Journal of Number
  Theory, 73 (1998), pp.~426 -- 450.

\bibitem{ajgiangreco2018cavff}
{\sc A.~J. Giangreco-Maidana}, {\em On the cyclicity of the rational points
  group of abelian varieties over finite fields}, arXiv preprint
  arXiv:1806.10842,  (2018).

\bibitem{Holden2004}
{\sc J.~Holden}, {\em Abelian varieties over finite fields with a specified
  characteristic polynomial modulo $\ell$}, Journal de Th{\'{e}}orie des
  Nombres de Bordeaux, 16 (2004), pp.~173--178.

\bibitem{honda1968}
{\sc T.~HONDA}, {\em Isogeny classes of abelian varieties over finite fields},
  J. Math. Soc. Japan, 20 (1968), pp.~83--95.

\bibitem{mumford1970abelian}
{\sc D.~Mumford}, {\em Abelian varieties}, vol.~5 of Tata Institute of
  fundamental research studies in mathematics, Oxford University Press, 1970.

\bibitem{ruck1987note}
{\sc H.-G. R\"{u}ck}, {\em A note on elliptic curves over finite fields},
  Mathematics of Computation, 49 (1987), pp.~301--304.

\bibitem{Rybakov2010}
{\sc S.~Rybakov}, {\em The groups of points on abelian varieties over finite
  fields}, Central European Journal of Mathematics, 8 (2010), pp.~282--288.

\bibitem{SCHOOF1987183}
{\sc R.~Schoof}, {\em Nonsingular plane cubic curves over finite fields},
  Journal of Combinatorial Theory, Series A, 46 (1987), pp.~183 -- 211.

\bibitem{tate1966}
{\sc J.~Tate}, {\em Endomorphisms of abelian varieties over finite fields},
  Invent. Math., 2 (1966), pp.~134--144.

\bibitem{tsfasman1985group}
{\sc M.~A. Tsfasman}, {\em Group of points of an elliptic curve over a finite
  field}, Theory of numbers and its applications, {Tbilisi},  (1985),
  pp.~286--287.

\bibitem{VLADUT199913}
{\sc S.~G. Vl\u{a}du\c{t}}, {\em Cyclicity statistics for elliptic curves over
  finite fields}, Finite Fields and Their Applications, 5 (1999), pp.~13 -- 25.

\bibitem{VLADUT1999354}
\leavevmode\vrule height 2pt depth -1.6pt width 23pt, {\em On the cyclicity of
  elliptic curves over finite field extensions}, Finite Fields and Their
  Applications, 5 (1999), pp.~354 -- 363.

\bibitem{voloch1988note}
{\sc J.~F. Voloch}, {\em A note on elliptic curves over finite fields}, Bull.
  Soc. Math. France, 116 (1988), pp.~455--458.

\bibitem{Waterhouse1969}
{\sc W.~C. Waterhouse}, {\em Abelian varieties over finite fields}, Annales
  scientifiques de l'\'Ecole Normale Sup\'erieure, 2 (1969), pp.~521--560.

\end{thebibliography}
\bibliographystyle{siam}

\end{document}